\documentclass[a4paper]{article}

\usepackage{
amsmath,
amsthm,
amscd,
amssymb,
wasysym
}
\usepackage{comment}
\usepackage{tikz}
\usepackage{stmaryrd}
\usetikzlibrary{positioning,arrows,calc}
\usepackage{xspace}
\usepackage{thmtools}
\usepackage{thm-restate}

\usepackage{graphicx}

\usepackage{url}

\theoremstyle{plain}
\newtheorem{theorem}{Theorem}
\newtheorem{lemma}[theorem]{Lemma}
\newtheorem{definition}[theorem]{Definition}

\newtheoremstyle{derp}
{3pt}
{3pt}
{}
{}
{\upshape}
{:}
{.5em}
{}
\theoremstyle{derp}
\newtheorem{example}{Example}

\newcommand{\Z}{\mathbb{Z}}

\newcommand{\N}{\mathbb{N}}

\newcommand{\lang}{\mathcal{L}}

\newcommand{\OC}[1]{\overline{\mathcal{O}(#1)}}

\newcommand\xqed[1]{%
  \leavevmode\unskip\penalty9999 \hbox{}\nobreak\hfill
  \quad\hbox{#1}}
\newcommand\qee{\xqed{$\fullmoon$}}

\newcommand{\EP}{\mathrm{EP}}

\newcommand{\sph}{\mathbb{S}}

\title{Entropy pair realization}

\author{
Ville Salo\thanks{Research supported by the Academy of Finland grant 2608073211.} \\
vosalo@utu.fi
}

\begin{document}
\maketitle

\begin{abstract}
We show that the CPE class $\alpha$ of Barbieri and Garc\'ia-Ramos contains a one-dimensional subshift for all countable ordinals $\alpha$, i.e.\ the process of alternating topological and transitive closure on the entropy pairs relation of a subshift can end on an arbitrary ordinal. This is the composition of three constructions: We first realize every ordinal as the length of an abstract ``close-up'' process on a countable compact space. Next, we realize any abstract process on a compact zero-dimensional metrizable space as the process started from a shift-invariant relation on a subshift, the crucial construction being the implementation of every compact metrizable zero-dimensional space as an open invariant quotient of a subshift. Finally we realize any shift-invariant relation $E$ on a subshift $X$ as the entropy pair relation of a supershift $Y \supset X$, and under strong technical assumptions we can make the CPE process on $Y$ end on the same ordinal as the close-up process of~$E$.
\end{abstract}

\section{Introduction}

The \emph{entropy} of a subshift $X$ is the limit of $\frac{\log |\lang_n(X)|}{n}$ where $\lang_n(X)$ is the set of words in $X$ of length $n$, as $n \rightarrow \infty$. It measures the amount of information in orbits of this dynamical system. There are several ways to quantify how this information is produced, one is the entropy pairs relation, due to Blanchard \cite{Bl92,Bl93}, which intuitively measures the amount of information in orbits when we obtain information only when the point is very close to one of two points (the elements of the entropy pair).

The entropy pairs relation (plus the diagonal) is symmetric, reflexive and topologically closed but not necessarily transitive, and a subshift turns out to have only positive entropy (non-trivial) factors if and only if the smallest closed equivalence relation containing the entropy pairs relation is the full relation \cite{Bl93}. The hierarchy generated by alternately closing the relation topologically and transitively is studied in \cite{BaGa18}, and it is shown that the full relation can arise at any countable ordinal. We refer to \cite{BaGa18} for the state-of-the-art and further context.

Our main contribution is that such examples can be made expansive. We prove this using relatively general subshift constructions, in particular Lemma~\ref{lem:EqRelsAsEntPairs} and Proposition~\ref{prop:AnyRelation} construct a large class of relations as entropy pair relations, which may be of independent interest.

By alternating topological and transitive closure on a symmetric reflexive relation $E$, one obtains a sequence of equivalence relations, and the union of these equivalence relations is of course just the smallest closed equivalence relation containing $E$.

\begin{definition}
\label{def:CloseUpRank}
Let $X$ be a topological space and $E \subset X^2$ a symmetric reflexive relation. If $E$ is closed, let $\zeta = 1$, otherwise $\zeta = 0$. Define $E^{(\zeta)} = E$ and write any ordinal greater than $\zeta$ in the unique way as $\lambda + 2n + b$, $\lambda$ a limit ordinal (including $0$), $n \in \N, b \in \{0,1\}$ and define
\begin{align*}
E^{(\lambda + 2n + 1)} &= \overline{E^{(\lambda + 2n)}} & {\mbox{(topological closure)}}  \\
E^{(\lambda + 2(n + 1))} &= (E^{(\lambda + 2n + 1)})^* & \mbox{(transitive closure)} \\
E^{(\lim_i \lambda_i)} &= \bigcup_{i} E^{(\lambda_i)} &
\end{align*}
We say the \emph{close-up rank} of $E$ is the least $\lambda$ such that $E^{(\lambda)} = E^{(\lambda+1)}$. We say $E$ is \emph{equalizing} if $E^{(\lambda)} = X^2$ for some $\lambda$.
\end{definition}

Ordinals $\lambda + 2n + b$ where $b = 0$ are called \emph{even}, and others \emph{odd}.

The check for topological closedness is made in the beginning so that $E^{(\lambda)}$ is always topologically closed at odd ordinals.

\begin{lemma}
\label{lem:WeakTopVersion}
Let $\lambda \geq 1$ be a countable ordinal. Then there exists a closed countable set $X \subset [0,1]$ and an equalizing closed symmetric reflexive relation $E \subset X^2$ with rank $\lambda$.
\end{lemma}

For this, we arrange an ordinal of the form $\omega^{\alpha}$ or $\omega^{\alpha} + 1$ on the circle and join each point with its successor. The calculation of the close-up rank roughly then corresponds to the computation of the Hausdorff or Cantor-Bendixson rank of the ordinal.

In \cite{BaGa18}, Barbieri and Garc\'ia-Ramos study the close-up process for the closed equivalence relation $\EP(X) = E \cup \Delta$ where $\Delta$ is the diagonal and $E$ is the family of \emph{entropy pairs} of the compact topological dynamical system $(G, X)$ ($G$ an amenable discrete group acting on $X$), where $(x, y)$ is an \emph{entropy pair} if the open cover $(U^c, V^c)$ has positive entropy \cite[Definition~14.5]{DeGrSi76} for all disjoint closed neighborhoods $U \ni x, V \ni y$.


In \cite{BaGa18}, a dynamical system is defined to be of \emph{CPE class $\alpha$} if the equivalence relation $E \cup \Delta$ is equalizing with rank $\alpha$. One of the results of \cite{BaGa18} is that among abstract dynamical systems, the CPE class $\alpha$ is nonempty for every countable amenable group $G$.

Nishant Chandgotia asked in personal communication if there is a subshift realization of their process, i.e.\ whether there is a subshift in the CPE class $\alpha$. We prove a positive answer for $G = \Z$. 

\begin{restatable}{theorem}{main}
\label{thm:Main}
Let $\alpha \geq 1$ be a countable ordinal. Then there exists a subshift $X \subset A^\Z$ in CPE class $\alpha$.
\end{restatable}

We prove this by first implementing closed relations on topological spaces as shift-invariant relations on subshifts, and then implementing such relations as CPE pairs.

A point $x \in A^\Z$ is \emph{Toeplitz} is $\forall i \in \Z: \exists n > 0: \forall m \in \Z: x_{i+mn} = x_i$.\footnote{Often one excludes the periodic points; indeed in our application the points are never periodic, but we omit this requirement so we need not address it explicitly.} A~\emph{Toeplitz subshift} is a subshift generated by a Toeplitz point. A~ \emph{pointwise zero-entropy Toeplitz} subshift is one where every point generates a Toeplitz subshift with zero entropy. Observe that by the variational principle \cite{DeGrSi76} a pointwise zero-entropy Toeplitz subshift has zero-entropy, since any generic point for an ergodic measure of positive entropy would obviously generate a Toeplitz subshift with positive entropy.

If $X$ is a topological space, the \emph{hit-or-miss topology} on its closed subsets has subbasis
\[ \mathcal{U}_{U,K} = \{ Y \subset X \mbox{ closed} \;|\; Y\cap K = \emptyset, Y \cap U \neq \emptyset \} \]
where $U$ ranges over open sets of $X$ and $K$ over compact sets of $X$. For the space of closed subsets of a subshift, this is equal to the topology given by Hausdorff distance \cite{Ku66}, and compares the set of words up to length $n$. This is the topology we use for the space of subshifts of a subshift $X$. The minimal subshifts form a subspace,\footnote{This subspace need not be closed, consider for example the space of minimal subshifts of the Grand Sturmian subshift \cite{Ku03}, where obviously no sequence of Sturmian subshifts with density of $1$s decreasing to zero has a converging subsequence.} and we give it the induced topology.

\begin{restatable}{lemma}{spacesassubshifts}
\label{lem:SpacesAsSubshifts}
For any compact metrizable zero-dimensional space $Z$, there exists a pointwise zero-entropy Toeplitz subshift $X$ and an open quotient map $\phi : X \to Z$ such that
$\phi^{-1}(z)$ is a minimal subshift for every $z \in Z$,
and $\phi$ induces a homeomorphism between $Z$ and the space of minimal subshifts of $X$. 
\end{restatable}

The construction is not new (see e.g.\ \cite{Ba13} for a more general construction in the recursive category), although we do not know a reference that checks the stated topological dynamical properties. The important word in the above statement is ``open'' -- intuitively because taking preimages of relations in an open quotient map induces a ``close-up process homomorphism'', in particular the close-up rank is preserved. 

We show that any close-up process on a subshift, satisfying some strong technical assumptions, can be realized as the entropy pairs process. We say a relation $E \subset X^2$ is \emph{doubly shift-invariant} if it is invariant under the action of $\Z^2$ that shifts the two components separately. If $\lambda$ is a successor ordinal, write $\lambda-1$ for its predecessor.

If $E \subset X^2$ is a relation, write $E^{[k]} = \{(x_1,x_2,\cdots,x_k) \;|\; \forall i,j: (x_i,x_j) \in E\}$ (so $E = E^{[2]}$ if and only if $E$ is symmetric). 

\begin{restatable}{lemma}{eqrelasentpairs}
\label{lem:EqRelsAsEntPairs}
Let $X \subset A^\Z$ be a subshift with zero entropy and let $E \subset X^2$ be a closed doubly shift-invariant symmetric reflexive relation. Then there exists a subshift $Y$ with $X \subset Y \subset \hat A^\Z \;(\supset A^\Z)$ such that $\EP(Y)^{(\lambda)} \cap X^2 = E^{(\lambda)}$ for all $\lambda$. Furthermore, if $E^{(\lambda)} \cap X^2 = X^2$, then $\EP(Y)^{(\lambda+1)} = Y^2$, and we have $\EP(Y)^{(\lambda)} = Y^2$ if
\begin{itemize}
\item $\lambda$ is even, or
\item $\lambda$ is odd and $(E^{[k]})^2 \subset \overline{(E^{(\lambda-1)})^{[2k]} \cap (E^{[k]})^2}$.
\end{itemize}
\end{restatable}

The intuitive description is that we add an entropy-generating supersystem on top of $X$, by allowing alternation between words from the two points from $X$ (separated by a special symbol) if and only if they are in the relation $E$. We want that $X$ originally has no entropy pairs, so we require it has zero entropy.

For $A \in X^k$ and $B \in X^k$, we can think of the pair $(A, B)$ as an element $(A, B) \in X^{2k}$ in a natural way.
For the intuition behind the \emph{odd condition} in Lemma~\ref{lem:EqRelsAsEntPairs} for odd $\lambda > 2$, first note that
$(E^{[k]})^2 \subset \overline{(E^{(\lambda-1)})^{[2k]} \cap (E^{[k]})^2}$
means precisely that if $A, B \in E^{[k]}$ then either $(a,b) \in E^{(\lambda-1)}$ for some $a \in A$ and $b \in B$ (equivalently $(A, B) \in (E^{(\lambda-1)})^{[2k]}$ by the transitivity of $E^{(\lambda-1)}$), or we can find $A_i, B_i \in E^{[k]}$ such that $(A_i, B_i) \in (E^{(\lambda-1)})^{[2k]}$ and $(A_i, B_i) \rightarrow (A, B)$.

This condition arises from the construction method: In the subshift $Y$, points encode finite tuples of points from $X$ any pair of which is in the relation. Elements of $Y$ thus roughly correspond to elements of $E^{[k]}$, and we want that at the final limit step, any pair of points in $Y$ (roughly corresponding to an element of $(E^{[k]})^2$) can be approximated by a pair of points from $Y$ which are in the relation at the previous step, which in terms of the close-up process of $E$ translates to being in $(E^{(\lambda-1)})^{[2k]}$. Not every relation satisfies the odd condition, see Example~\ref{ex:NotEvery}.

This odd condition holds in the construction proving Lemma~\ref{lem:WeakTopVersion} (see Lemma~\ref{lem:StrongTopVersion} for the precise statement), and is preserved under open quotients, so it passes to the subshift implementations of abstract processes given by Lemma~\ref{lem:SpacesAsSubshifts}.


We do not know to what extent doubly shift-invariance, the zero-entropy assumption and the odd condition can be weakened in the previous lemma. Out of general interest, for implementing any relation as entropy pairs of a supersystem (with no control on the ``new pairs''), we provide a simpler construction.

\begin{restatable}{proposition}{anyrelation}
\label{prop:AnyRelation}
Suppose $X$ has zero entropy and $E \subset X^2$ is a closed symmetric reflexive shift-invariant relation. Then there exists a subshift $Y$ with $X \subset Y \subset \hat A^\Z \;(\supset A^\Z)$ such that $\EP(Y) \cap X^2 = E$.
\end{restatable}

This is optimal in the sense that for positive entropy $X$, it fails for $E$ the diagonal relation \cite{Bl93,KeLi07,BaGa18}, and for $X$ zero-entropy and for any $Y \supset X$, $\EP(Y) \cap X^2$ is closed symmetric reflexive shift-invariant. One could, however, ask for the aesthetically more pleasing $\EP(Y) = E \cup \Delta_Y$ instead of just $\EP(Y) \cap X^2 = E$, and we conjecture that this can be done. Even further, one could consider $X$ with have positive entropy, and restrict to relations containing $\EP(X)$.

The constructions are made for $\Z$, and of course one can ask whether they can be generalized to general amenable groups. We conjecture that all CPE classes are inhabited by subshifts on all countable amenable groups.

Particlarly in the case of $\Z^d$, one can also ask if this can be done by SFTs. We conjecture that $\Z^2$-SFTs inhabit exactly the CPE classes $\alpha$ where $\alpha$ is recursive.\footnote{This has been proved by Linda Westrick \cite{We19}.}

\section{Definitions}

A \emph{topological dynamical system} is $(X,f)$ where $X$ is a compact metrizable space and $f : X \to X$ a homeomorphism. The \emph{full shift} is the dynamical system $(A^\Z, \sigma)$, where $A$ is a finite set and $A^\Z$ has the (Cantor) product topology, and $\sigma(x)_i = x_{i+1}$ is the \emph{left shift}. A \emph{subshift} is a closed $\sigma$-invariant subsystem of the full shift. A subshift is \emph{minimal} if it does not contain nonempty proper subshifts. Topological dynamical systems form a category with morphisms the dynamics-commuting continuous maps. A \emph{cover} is a surjective morphism or its codomain. A cover is \emph{almost $1$-to-$1$} if the set of points with a unique preimage is dense. See any standard reference for the basic theory and definitions related to topological entropy \cite{DeGrSi76}.
 
Words -- finite or infinite -- are functions from contiguous subsets of $\Z$ to a (by default finite) alphabet. We typically call infinite words \emph{points}, as we think of them as points of the full shift dynamical system (or another subshift). Positioning conventions should always be clear from context, or are explained when used. Concatenation of words is presented by $u\cdot v$ or $uv$. A point $x \in A^\N$ is \emph{eventually constant} if $\exists n: \forall i \geq n: x_i = x_{i+1}$.

A \emph{(deterministic) substitution} is a map $\tau : A \to B^*$ where $A,B$ are (possibly infinite) alphabet, and $B^*$ denotes the set of finite words over $B$, including the empty word. We can apply a substitution to a word or a point by replacing the individual symbols $a \in A$ by words $\tau(a)$. A \emph{nondeterministic substitution} is a function $\tau : A \to \mathcal{P}(B^*)$ where $\mathcal{P}$ denotes the power set. Applying a nondeterministic substitution to a word or point (or a set of words or points) results in the set of all words or points obtainable by replacing each symbol $a \in A$ by one of its images $w \in \tau(a)$.

\section{The proofs}

\subsection{Close-up processes of arbitrary length}

See any standard reference for definitions and basic properties of ordinal arithmetic (e.g.\ \cite{HrJe99}). The following lemma is roughly just the computation of the Hausdorff rank of a (cyclic) order corresponding to an ordinal. Since the definition of the close-up process involves parity and the restriction to compact spaces leads to a small issue with non-compact ordinals, it seems easier to give a direct proof than to find a reference.

\begin{lemma}
\label{lem:StrongTopVersion}
Let $\lambda \geq 1$ be a countable ordinal. Then there exists a closed countable set $X \subset [0,1]$ and an equalizing closed symmetric reflexive relation $E \subset X^2$ with rank $\lambda$. Furthermore, if $\lambda > 2$ is an odd ordinal, then
$(E^{[k]})^2 \subset \overline{(E^{(\lambda-1)})^{[2k]} \cap (E^{[k]})^2}$.
\end{lemma}

\begin{proof}
To each $\alpha = \omega^{\zeta}$ or $\alpha = \omega^{\zeta}+1$ for $\zeta \geq 1$ we associate a countable $X \subset [0,1]$, such that for $\zeta = \omega^{\lambda + n}$ the close-up rank of $E$ is $\lambda + 2n$, and for $\zeta = \omega^{\lambda + n} + 1$ it is $\lambda + 2n + 1$.

Let $\alpha$ be as above and let $A = \{0,1,...,\alpha\}$ with the order topology. As a Von Neumann ordinal, $A = \alpha+1$, but the intended reading of $\phi(A)$ is $\phi(A) = \{\phi(a) \;|\; a \in A\}$. Take an injective order homomorphism $\phi : A \to [0,1]$, so that $\phi(0) = 0$, $\phi(\alpha) = 1$. Identify $0 \equiv 1$ to obtain a continuous function $\phi : A \to \phi(A) \subset \sph^1 = [0,1]/{\equiv}$, and let $\phi(A) = X$. The reason for arranging the ordinal on a sphere is simply to compactify the non-compact ordinals, and the reason for wanting non-compact ordinals is to allow the last step of the close-up process to be a transitive closure step.

While we do the construction on $\sph^1 = [0,1]/{\equiv}$ rather than $[0,1]$, we note that any non-full closed subset of $\sph^1$ is homeomorphic to a closed subset of $[0,1]$.

Let $E \subset (\sph^1)^2$ be the union of the diagonal relation $\Delta$ and the successor and predecessor relations, i.e.
\begin{align*}
E = & \{(\phi(\beta), \phi(\beta)), (\phi(\beta), \phi(\beta+1)), (\phi(\beta+1), \phi(\beta)) \;|\; \beta+1 < \alpha \}
\end{align*}
Observe that if $\alpha$ is a successor, we do not join $\phi(\alpha) = \phi(0)$ with its ``predecessor'' $\phi(\alpha-1)$. It is easy to see that $X$ and $E$ are closed. Consider the close-up process for this relation, starting from $E^{(1)} = E$.

Note that at all times, $E^{(\lambda)}$ for $\lambda \geq 1$ is symmetric and reflexive, so we simply have to figure out which pairs $(\phi(\beta), \phi(\beta + \gamma))$ are contained in $E^{(\lambda)}$ for each ordinal $\lambda \geq 1$. What happens is that only the length of the jumps $\gamma$ matters. Alternately, we are allowed to first take less than $\omega$ steps with the current relation (a transitive closure step), and then to close up these jumps into a single superjump consisting of $\omega$ many previous jumps (a topological closure step).

Concretely, we prove an exact characterization of $E^{(\lambda)}$ by transfinite induction: Every ordinal can be written in a unique way as $\lambda + 2n$ (even ordinals) or $\lambda + 2n + 1$ (odd ordinals) where $n \in \N$ and $\lambda$ is a limit ordinal (interpreting $0$ as a limit ordinal). For all $\lambda, n$, we prove by induction for all ordinals $\lambda+2n+b$ greater than $0$ 
($n \in \N, b \in \{0,1\}$) that 
\[ E^{(\lambda + 2n)} = \{ (\phi(\beta), \phi(\beta + \gamma)), (\phi(\beta + \gamma), \phi(\beta)) \;|\; \gamma < \omega^{\lambda + n}, \beta + \gamma < \alpha \}, \]
\[ E^{(\lambda + 2n + 1)} = \{ (\phi(\beta), \phi(\beta + \gamma)), (\phi(\beta + \gamma), \phi(\beta)) \;|\; \gamma \leq \omega^{\lambda + n}, \beta + \gamma < \alpha \}. \]


The claim for $E^{(1)}$ is (setting $\lambda = 0, n = 0, b = 1$) equivalent to having (up to symmetry) $E^{(1)} \ni (\beta, \beta+\gamma)$ if and only if $\gamma < 2$, $\beta + \gamma < \alpha$. This is precisely the definition of $E^{(1)} = E$.

For the induction steps, suppose $E^{(\lambda+2n)}$ is of the claimed form, and consider $E^{(\lambda+2n+1)}$. This is a topological closure step, i.e.\ $E^{(\lambda+2n+1)} = \overline{E^{(\lambda+2n)}}$. Taking the closure of the pairs given by the inductive assumption, we clearly have
\[ \{ (\phi(\beta), \phi(\beta + \gamma)) \;|\; \gamma \leq \omega^{\lambda + n}, \beta + \gamma < \alpha \} \subset E^{(\lambda+2n+1)}. \]

If $\beta > 0$, then $(\phi(\beta), \phi(\beta + \gamma)) \notin E^{(\lambda+2n+1)}$ for $\gamma \geq \omega^{\lambda + n}+1$, $\beta + \gamma < \alpha$ since otherwise $(\phi(\theta), \phi(\beta+\gamma)) \in E^{(\lambda+2n)}$ for $\theta \leq \beta$ arbitrarily close to $\beta$ and $\gamma \geq \omega^{\lambda + n}+1$ (note that $\beta + \omega^{\lambda + n}+1$ is isolated), contradicting the inductive assumption.

If $\beta = 0$, then there may be additional limits on the right end of $[0,1]/{\equiv}$. In this case suppose we have $(\phi(\beta), \phi(\beta + \gamma)) = (0, \phi(\gamma)) \in E^{(\lambda+2n+1)}$ for $\omega^{\lambda + n}+1 \leq \gamma < \alpha$. To obtain such pairs, we must have $(\phi(\delta), \phi(\theta)) \in E^{(\lambda+2n)}$ for $\theta < \alpha$ arbitrarily close to $\alpha$ and $\delta \leq \gamma$ arbitrarily close to $\gamma$. If $\alpha = \omega^{\zeta}+1$, then $\alpha$ cannot be approximated from below, so this is impossible.

If $\alpha = \omega^{\zeta}$, then observe that since $\gamma < \alpha$, we have $\omega^{\lambda + n}+1 \leq \gamma < \alpha$, so since $\alpha = \omega^{\zeta}$, we have $\omega^{\lambda + n}+1 \leq \gamma \leq \omega^{\zeta'} \cdot m$ for some $\zeta' < \zeta$, $m \in \N$ (here we simplify the successor and limit case of $\zeta$ into one). Observe that we even have $\omega^{\zeta'} \cdot \omega \leq \alpha$. Now, for any $\delta \leq \gamma$ and $\omega^{\zeta'} \cdot 2m \leq \theta < \alpha$ we have
\[ \delta + \omega^{\lambda + n} + 1 \leq \omega^{\zeta'} \cdot 2m \leq \theta \]
by left distributivity and monotonicity of ordinal multiplication, which means $(\phi(\delta), \phi(\theta)) \notin E^{(\lambda+2n)}$, and thus $(\phi(\delta), \phi(\theta)) \notin \overline{E^{(\lambda+2n)}}$. This concludes the analysis of the step from $\lambda+2n$ to $\lambda+2n+1$.

Suppose then that $E^{(\lambda+2n+1)}$ is of the claimed form 
and consider $E^{(\lambda+2n+2)}$. Since $(\phi(\beta), \phi(\beta + \gamma)) \in E^{(\lambda+2n+1)}$ for all $\beta$, $\beta + \gamma < \alpha$ with $\gamma \leq \omega^{\lambda + n}$, taking the transitive closure we have
$(\phi(\beta), \phi(\beta + \gamma)) \in E^{(\lambda+2n+2)}$
for any $m \in \N$ and $\gamma \leq \omega^{\lambda + n} \cdot m$ with $\beta + \gamma < \alpha$. This is indeed equivalent to having $(\phi(\beta), \phi(\beta + \gamma)) \in E^{(\lambda+2n+2)}$ for any $\gamma < \omega^{\lambda + n} \cdot \omega = \omega^{\lambda + n + 1}$ as required.

To see that $E^{(\lambda+2n+2)}$ does not contain any pairs $(\phi(\beta), \phi(\beta + \gamma))$ with $\gamma > \omega^{\lambda + n} \cdot \omega$, $\beta + \gamma < \alpha$, consider any tuple $(\phi(\beta_1), \phi(\beta_2), \cdots, \phi(\beta_k))$ with $(\phi(\beta_i), \phi(\beta_{i+1})) \in E^{(\lambda+2n+1)}$, $\beta_i \neq \beta_{i+1}$ for all $i$, and $\beta_1 < \beta_k$. If $\beta_j$ is minimal among of the $\beta_i$, then it easily follows from the inductive assumption that
\[ \beta_k \leq \beta_j + \omega^{\lambda + n} \cdot k \leq \beta_1 + \omega^{\lambda + n} \cdot k < \beta_1 + \omega^{\lambda + n} \cdot \omega. \]
This concludes the analysis of the step from $\lambda+2n+1$ to $\lambda+2n+2$.

Finally, we look at the case that $\lambda$ is an infinite limit ordinal. By definition, $E^{(\lambda)} = \bigcup_{\delta < \lambda} E^{(\delta)}$. We need to show that $E^{(\lambda)}$ contains $(\phi(\beta), \phi(\beta+\gamma))$ for $\beta+\gamma < \alpha$ if and only if $\gamma < \omega^{\lambda}$. Since $\gamma < \omega^{\lambda}$ is equivalent to $\exists \delta < \lambda: \gamma < \omega^{\delta}$, this follows from the inductive assumption on the relations $E^{(\delta)}$. This concludes the analysis of the limit steps.

Now observe that by the above characterization, $E^{(\lambda+2n+b)} = X^2$ if and only if $(0,\gamma) \in E^{(\lambda+2n+b)}$ for all $\gamma < \alpha$. This happens if and only if $b = 0$ and $\gamma < \alpha \implies \gamma < \omega^{\lambda + n}$, i.e.\ $\alpha \leq \omega^{\lambda + n}$; or $b = 1$ and $\gamma < \alpha \implies \gamma \leq \omega^{\lambda + n}$, i.e.\ $\alpha \leq \omega^{\lambda + n} + 1$.

Now, to prove the original claim, for $\lambda = 1$ we can realize close-up rank $1$ by setting $E = X^2$ for any choice of $X$. For larger countable ordinals, in the above construction we obtained for the choice $\alpha = \omega^{\lambda + n}$ that the close-up rank of $E$ is $\lambda + 2n$, and that for the choice $\alpha = \omega^{\lambda + n}+1$ it is $\lambda + 2n + 1$.

For the last sentence, observe that in the case of odd $\lambda + 2n + 1$, the pairs added at the last step are just those joining some $\phi(\beta)$ for $\beta < \alpha-1$ with $\phi(\alpha-1)$. We can realize any such limit through pairs $(\phi(\beta), \phi(\beta_i)) \in E^{(\lambda+2n)}$ where $\beta_i \nearrow \alpha-1$.

Now consider any $(a_1,...,a_k,b_1,...,b_k) \in (E^{[k]})^2$, where we can take the $a_i$ to be sorted in increasing order of $\phi^{-1}(a_i)$, and similarly for the $b_i$. If $(a_1,b_1) \in E^{(\lambda+2n)}$ then
\[ (a_1,...,a_k,b_1,...,b_k) \in (E^{(\lambda+2n)})^{[2k]} \cap (E^{[k]})^2 \subset \overline{(E^{(\lambda+2n)})^{[2k]} \cap (E^{[k]})^2} \]
follows from transitivity of $E^{(\lambda+2n)}$. Otherwise by the above paragraph, and up to symmetry, we can assume $(a_1, b_1)$ is the limit of $(\phi(\beta), \phi(\beta_i))$. We have $b_i = b_1$ for all $i$ since $\phi(\alpha-1)$ is not in $E$-relation with any other point. We have $a_i \in \{ a_1, \phi(\phi^{-1}(a_1) + 1) \}$ for all $i$ (recall that the original relation $E$ is just the successor relation, symmetrized and reflexivized), and by the characterization of $E^{(\lambda+2n)}$ we then have $(a_i, \phi(\beta_i)) \in E^{(\lambda+2n)}$ for any $a_i$, as required.
\end{proof}

\subsection{Spaces as Toeplitz subshifts}

See e.g.\ \cite{Ku03} for basic information on Toeplitz subshifts. We recall a basic lemma about Toeplitz subshifts. This is a special case of \cite[Theorem~5.23]{GoHe55}.\footnote{Their term for the formula below is ``isochronous'', and their term for Toeplitz is ``regularly almost periodic''.}

\begin{lemma}
Suppose $x \in A^\Z$ satisfies
\[ \forall i \in \Z, k > 0: \exists j \in \Z: \exists n > 0: \forall \ell \in [-k,k], m \in \Z: x_{j+\ell+mn} = x_{i+\ell}. \]
Then $\overline{O(x)}$ is Toeplitz.
\end{lemma}

In words, the assumption is that every subword appearing in $x$ at $i$ appears in some arithmetic progression in $x$ (but not necessarily one going through position~$i$).

\begin{proof}
We need to show that the orbit-closure of $x$ contains a Toeplitz point with the same language as $x$. Let $w_1 = x|_{[-1,1}$ and let $j_1$ be such that an arithmetic progression of $w_1$'s with difference $n_1$ begins at $j_1$. Replace $x$ by $x^1 = \sigma^{j_1}(x)$, so that we still have $w_1 = x^1|_{[-1,1]}$ but now the central three coordinates are in an arithmetic progression. From now on, restrict to shifting by multiples of $n_1$.

By the original assumption, the central word $x^1|_{[-n_1-|j_1|-2, n_1+|j_1|+2]}$ is in an arithmetic progression beginning at some $j_2$ with difference $n_2$, we may take $n_1 \;|\; n_2$. Let $x^2 = \sigma^{h_2 n_1}(x^1)$ where $h_2 = j_2 \sslash n_1 \in \Z$ where $\sslash$ denotes integer division discarding remainder. Since we shifted by a multiple of $n_1$, in $x^2$ the word at $[-1,1]$ still lies in an arithmetic progression with difference $n_1$. But since we shifted by a number with distance at most $n_1$ from $j_2$, now the subword $x^2|_{[-2-|j_1|, 2+|j_1|]}$ lies is part of an arithmetic progression with difference $n_2$. Observe that this contains the word $x|_{[-2, 2]}$.

In $x^2$, we have inserted Toeplitz periods for more coordinates, without modifying the periods we already introduced for $x^1$, and also introduced larger subwords of $x$ into the periodic area. An obvious induction gives a Toeplitz point with the same language as that of $x$ in the limit.
\end{proof}

The \emph{positioned words} over $A$ are $A^{**} = \{w : [a,b] \to A \;|\; a,b \in \Z \}$. For $w \in A^{**}$, $w : [a,b] \to A$, and $X \subset A^\Z$ a subshift (deduced from context), write $[w] = \{x \in X \;|\; x|_{[a,b]} = w\}$.

\spacesassubshifts*

\begin{proof}
Define the subshift $X' \subset \{0,1,\#\}^\Z$ by requiring that in every point, there is an arithmetic progression $\{3k+i \;|\; k \in \Z\}$ containining only $\#$, such that $\{3k+i+1 \;|\; k \in \Z\}$ is a constant sequence containing only $0$ or only $1$, and recursively $\{3k+i+2 \;|\; k \in \Z\}$ is a point of $X'$.

This indeed describes a unique subshift whose typical points look like
\[ \ldots \#z_0\#\#z_0z_1\#z_0\#\#z_0\#\#z_0z_1\#z_0z_2\#z_0\#\#z_0z_1\#z_0\#\#\ldots \]
where $z \in \{0,1\}^\N$, and we associate a sequence of bits to every point in $X'$ by locating the unique infinite $3$-progression of $\#$s (which can be deduced from any word of length $3$), outputting the bit to the right of one (thus any) of them, and recursively extracting the rest of the bits from the third $3$-progression. Let $\phi : X' \to \{0,1\}^\N$ extract this sequence of bits, so $\phi$ is continuous, and $\phi(x) = \phi(\sigma(x))$ for all $x \in X'$.

Since $Z$ compact, metrizable and zero-dimensional, we can take (up to homeomorphism) $Z \subset \{0,1\}^\N$ a closed set such that no point in $Z$ is eventually constant, equivalently, $01$ appears infinitely many times in every $z \in Z$. Define $X = \phi^{-1}(Z)$, which is clearly a subshift. A simple computation shows that the number of words of length $n$ in $X$ (even in $X'$) is $O(n^{\log_3 6})$, so $X$ has zero entropy.

The map $\phi$ is surjective and continuous by definition, and it is closed since $X$ is compact and $Z$ Hausdorff, so it is a quotient map from $X$ to $Z$.

Let us now analyze the dynamical structure of $Z$. Call the unique infinite $3$-progression of $\#$s in a point of $X$ the \emph{$0$-skeleton}. Recursively we refer to the $0$-skeleton of the point of $X'$ obtained after $i$ iterations of the recursive extraction process as the $i$-skeleton. The \emph{skeleton} refers collectively to the family of all the $i$-skeletons, more precisely the skeleton of $x \in X$ is $\pi(x)$, where $\pi(\#) = \#$, $\pi(0) = \pi(1) = 0$.

Dynamically, extraction of the skeletons gives an almost-$1$-to-$1$ subshift cover of the $3$-odometer $(\Z_3, (a \mapsto a+1))$ (where $\Z_3$ is the space of $3$-adic integers) more precisely $\pi(X)$ is a $2$-to-$1$, and almost $1$-to-$1$, cover of the $3$-odometer, by the obvious map $\psi : \pi(X) \to \Z_3$ that records the offsets of the $\#$-progressions in an intertwining way.
What happens is that a typical $\psi(\pi(x))$ determines the skeleton $\pi(x)$ entirely. If it does not, then there is a single ``hole'' left after filling in the $\#$s. If there is a hole left, then $|\psi^{-1}(\psi(\pi(x)))| = 2$.

It is easy to see that if $\psi^{-1}(\psi(\pi(x)))$ is a singleton, then for any $z \in Z$, there is exactly one point $x' \in X$ with $\phi(x') = z$ and $\pi(x') = \pi(x)$, since after filling in the skeleton and the bits to the right of it, there is no hole left to insert a $\#$, thus no hole to insert a bit either. If there is a hole left, then we can insert any of $0,1,\#$, but all other cells are determined by $\psi(\pi(x))$ and $\phi(x)$.

It follows that the subshift $X$ itself is a $3$-to-$1$ and almost-$1$-to-$1$ subshift cover of the system $\Z_3 \times Z$ under $(\alpha, z) \mapsto (\alpha+1, z)$ where $\Z_3$ denotes the $3$-odometer.

We now show openness of $\phi$. Suppose $x \in [u]$ and $\phi(x) = z$. Pick $x' \in [u]$ such that $\phi(x') = z$ and
\[ \forall y, y': \pi(y') = \pi(y) = \pi(x') \wedge \phi(y) = \phi(y') \implies y = y'. \]
This is possible because $[u]$ determines only part of the skeleton of $x$, and we may modify the rest to find $x'$ such that $\pi(x')$ is a singleton $\psi$-fiber. The bits left in $u$ after determining part of the skeleton do not constrain this since there are no eventually constant points in $Z$; i.e.\ if there is a hole left in $x$ after determining the skeleton and filling in the bits of $\phi(z)$, and this hole is filled with $a \in \{0,1,\#\}$, then we can pick $x'$ so that this hole is filled with $a$, at the point of the process where a suitable bit emerges.

Now pick $w \in \{0,1\}^{*}$ with $z \in [w] \subset [v]$, long enough so the bits in $u$ are determined by those in $w$ (when using the skeleton of $x'$), so that for all $z' \in [w]$ there exists a (unique) point $y \in [u]$ with $\pi(y) = \pi(x') \wedge \phi(y) = z'$. 
This means $\phi([u]) \supset [w] \ni z$. Thus $\phi$ is open.

To see that every point generates a Toeplitz subshift, let $x \in X$ be arbitrary and suppose $x \in [u]$ for some $u \in A^{**}$, $u : [a,b] \to A$. If $\psi^{-1}(\psi(\pi(x)))$ is a singleton, i.e.\ $\pi(x)$ and $\phi(x)$ determine $x$ uniquely, then clearly there exists a period $3^n > 0$ such that $\sigma^{i3^n}(x) \in [u]$ for all $i \in \Z$, i.e.\ $x$ is already Toeplitz. In general, it suffices to show the condition from the previous lemma, so we show that the word $u$ appears in some arithmetic progression in $x$.


For this, note that since $\pi(X)$ is an almost-$1$-to-$1$ cover of $Z_3$, we can find $m$ such that the set $\psi^{-1}(\psi(\pi(\sigma^j(x))))$ has very small diameter, concretely meaning that the skeleton and bits to the right of $\#$s in the skeleton determine all of the coordinates in $\sigma^j(x)_{[a,b]}$. Since $\phi(x)$ is not eventually constant, we can pick $j$ so that together with the bits of $z$, the forced word is $\sigma^j(x)_{[a, b]} = u$. The subword $[j+a, j+b]$ of $x$ then appears in an arithmetic progression because of the equicontinuity properties of the $3$-odometer, as required.

Observe now that (since every Toeplitz subshift is minimal) every point generates a minimal subshift, and this subshift maps to a single point in $\phi$. Conversely, the language of a minimal subshift is determined by the bits in its image in a continuous way (again use that no point in $Z$ is constant). Thus the fibers of $\phi$ are precisely the minimal subshifts of $X$ and $\phi$ induces a homeomorphism between them and $Z$.
\end{proof}

\subsection{Relations as entropy pairs}

Next, we realize relations as entropy pairs. We first deal with the easy case where $E$ is not doubly shift-invariant, but we only realize the relation, not the whole process. We prove this statement as it sounds nicer than the technical result we actually need, and illustrates some of the ideas used in the following section. However, it is not used in the construction of CPE ranks.

For $y \in A^\Z$, write $w \sqsubset y \iff \exists i: y_{[i,i+|w|-1]} = w$, and for $Y \subset A^\Z$ write $w \sqsubset Y \iff \exists y \in Y: w \sqsubset y$.

\anyrelation*

\begin{proof}
Let $\hat A = A \cup \{\#\}$, fix a sequence $(x^n, y^n) \in E$ such that each pair $(x, y) \in E$ is a  limit point of the sequence, and define
\[ Y_n = {\{ \ldots \# u_i \# u_{i+1} \# u_{i+2} \# \ldots \;|\; \forall i: u_i \in \{x^n_{[-n,n]}, y^n_{[-n,n]}\} \}}  \]
(i.e.\ $Y_n$ is the shift-invariant set of all points of the stated form). Let $Y = \overline{\bigcup_{n} Y_n}$.

If $(x,y) \in E \setminus \Delta$, then obviously\footnote{This is obvious from the intuition stated in the introduction: even if we only obtain information when the orbit is very close to $x$ and $y$, we obtain an exponential amount of information from points of $Y_n$ if $d((x_n, y_n), (x, y))$ is very small. The concrete computation is also straightforward.} by picking larger and larger $n$ such that $(x^n, y^n)$ is very close to $(x,y)$, $Y_n \subset Y$ implies that $(x,y) \in \EP(Y)$.

Suppose then that $(x,y) \in X^2 \setminus E$. Since $E$ is closed, for some $n$ we have $([x_{[-k,k]}] \times [y_{[-k,k]}]) \cap E = \emptyset$. We claim that for $U = [x_{[-k,k]}], V = [y_{[-k,k]}]$, the cover $(U^c, V^c)$ has zero topological entropy. 

It is enough to give, as a function of $m$, a set of binary words $U \subset \{0,1\}^m$, whose size grows subexponentially in $m$, such that for any $w \in \lang_m(Y)$, there exists $u \in U$ such that $(U^c, V^c)^u$ intersects\footnote{Covering partial orbits is more difficult than covering words in this sense by at most a constant factor, so this is indeed sufficient.} $[w]$, where $(U^c, V^c)^u$ is defined inductively on word length by
\[ (U^c, V^c)^0 = U^c, \;\; (U^c, V^c)^1 = V^c, \,\; (U^c, V^c)^{a \cdot u} = (U^c, V^c)^a \cap \sigma^{-1}((U^c, V^c)^u) \]
for $a \in \{0,1\}, u \in \{0,1\}^m$. Say $u \in \{0,1\}^m$ \emph{covers} $w \sqsubset Y$ if $(U^c, V^c)^u \cap [w] \neq \emptyset$.

We first note that the words $w$ containing at most one $\#$ are easy to cover. For such $w \sqsubset Y$, we always have $w \sqsubset X$ or $w = u\#v$ where $u$ and $v$ are words in $X$. All such words appear in a morphic image of the zero entropy subshift $X^2 \times \OC{...000111...}$, so they cannot generate entropy w.r.t.\ any cover.

To cover words with at least two $\#$-symbols, observe that in $Y$ all such words have an arithmetic progression of $\#$s. Consider a progression $\alpha = \{i(2n+2) + c \;|\; i \in \Z\} \cap \{0,1, 2,..., m-1\}$ where $n \geq 0, c \in [0,2n+1]$ and $|\alpha| \geq 2$. Then the set of words where $\#$s appear exactly in this arithmetic progression consists precisely of words obtained by inserting any combination of $x^n_{[-n,n]}$ or $y^n_{[-n,n]}$ between $i(2n+2) + c$ and $(i+1)(2n+2) + c$ for each $i$ (and we have up to $4$ choices in total for the prefix and suffix).

Since $(x^n,y^n) \in E$, by shift-invariance also $(\sigma^\ell(x^n), \sigma^\ell(y^n)) \in E$ for all $\ell \in \Z$ and thus from $([x_{[-k,k]}] \times [y_{[-k,k]}]) \cap E = \emptyset$ we obtain that for all $\ell$,
\[ \{x^n_{[\ell-k,\ell+k]}, y^n_{[\ell-k,\ell+k]}\} \neq \{x_{[-k,k]}, y_{[-k,k]} \}, \]
and thus see that there is a single word $u_{n,c}$ of length $m$ that covers every word $w$ with progression $\alpha$. This gives a polynomial bound on the number of words $u$ needed to cover those $w$ with at least two $\#$s, all in all we have obtained a subexponential bound.
\end{proof}

\subsection{Doubly invariant relations as entropy pairs}

\eqrelasentpairs*

\begin{proof}
Define $w_0 = 1$ and inductively $w_{i+1} = w_i 0^{i+1} w_i$. Define
\[ w_\omega = \lim_i w_i \in \{0,1\}^\N = 101001010001010010100001010010100010100101... \]
Another description is that $w_\omega$ is obtained from the \emph{ruler sequence} A007814 in OEIS \cite{OEISruler} or \emph{universal counterexample} \cite{Fe06}
\[ 01020103010201040102010301020105... \]
by applying the substitution $n \mapsto 10^{n+1}$ (or $0 \mapsto 1$, $n \mapsto 0^{n}$). We refer to the words $w_i$ and their obvious positioned variants as \emph{full words} and refer to the process of extending a positioned $w_i \in A^{**}$ to $w_{i+1} \in A^{**}$ (which can be done in exactly two ways) as the \emph{left} or $\emph{right full extension}$. 

By induction $n_i = |w_i| = 3 \cdot 2^i - i - 2$ and $\sum_j (w_i)_j = 2^i$. A short computation shows that for every $k \geq 1$, there exists a subword of $w_\omega$ of length $k$ (namely its prefix) where the symbol $1$ appears at least $k/3$ times (any $k/c$ works just as well for what follows).

We now perform some substitutions to obtain a subshift. Let $\tau(0) = \{0\}, \tau(1) = \{1, 11\}$ be a nondeterministic substitution and define $Y''' = \tau(w_\omega) \subset \{0,1\}^\N$, i.e.\ $Y'''$ is the set of all possible points obtained by replacing $0$s in $w_\omega$ by $0$s, and $1$s by either $1$ or $11$. Let $\hat A = A \sqcup \{\#\}$ and let $\tau' : \hat A \to \{0,1\}^*$ be the deterministic substitution mapping $\# \mapsto 1$, $a \mapsto 0$ for $a \in A$ and $\# \notin A$, and let $Y'' = (\tau')^{-1}(Y''')$. Let $Y'$ be the $\Z$-subshift whose language consists of the left-extendable words in the language of $Y''$. Let $Y$ be the subshift of $Y'$ with additional forbidden words
\[ \{ u_0\#u_1\#\cdots\#u_k \;|\; \not\exists x_0, x_1, \cdots x_k \in X:\forall i, j: x_i \in [u_i] \wedge (x_i, x_j) \in E\} \]
(where the $u_i$ are quantified over $A^*$). The double occurrences of $\#$ fit this pattern by picking empty words $u_i$.

Say that a set of words $U \subset A^*$ is a \emph{friendship} if for any finite subset $V \subset U$ there exist points $(b_v)_v \in X^V$ such that $\forall u, v \in V: b_u \in [u] \wedge (b_u, b_v) \in E$, so since $E$ is doubly shift-invariant, the forbidden words of $Y$ precisely require that the set of words $U \subset A^*$ that appear in $y$ is a friendship. (We remark that even for a closed doubly shift-invariant equivalence relation, for a finite set $U$, in general being a friendship is stronger than all pairs $\{u, v\}$ being friendships for $u, v \in U$.)

For a concrete picture of points in $Y'$, consider the sequence $w_\omega \in \{0,1\}^\N$. 
Its two-sided orbit closure contains points of the form
\[ ...101001010001010010100001010010100010100101... \]
We allow any of the $1$s to be duplicated, e.g.\ 
\[ ...101100101000110110010110000101001101000101001011..., \]
and then replace the $1$s by $\#$, i.e.\ 
\[ ...\#0\#\#00\#0\#000\#\#0\#\#00\#0\#\#0000\#0\#00\#\#0\#000\#0\#00\#0\#\#... \]
finally replace maximal $0$-sequences by words over $A$ of the same length, e.g.
\[ ...\#0\#\#01\#0\#011\#\#1\#\#10\#1\#\#0100\#1\#01\#\#1\#101\#1\#11\#1\#\#... \]
if $A = \{0,1\}$. Such are the points of $Y'$. In points of $Y$, the finite words should additionally be from the language of $X$, and should form a friendship.

If $U \subset A^*$ is a friendship, then by compactness to every word $u \in U$, we can associate a point $b_u \in [u]$, so that $(b_u, b_v) \in E$ for all $u, v \in U$. Fix such points for each set $U$ and call $b_u$ the \emph{(friendship) bracelet of $u$ (w.r.t. $U$)}. Of particular importance in what follows are (sets of) bracelets for sets of words $U$, not containing $\#$, that appear in some point $y \in Y$, or the union of such sets for two points $y, z$. We note that in such a situation, the bracelets are not enforced to be in the orbit closure of $y$ or $z$.

Call the binary word $\tau'(w)$ recording positions of $\#$ by $1$ in a point $w$ the \emph{skeleton}, similarly for points $x \in Y$, and the word where $11$ is further replaced by $1$ the \emph{preskeleton}. We call a (positioned or not) word $w \in \hat A^* \cup \hat A^{**}$ \emph{full} if it is equal to (a shift of) $(\tau')^{-1}(w_n)$ for some $n$, or if (more generally) its preskeleton is full.

We now explain the technical trick for characterizing the entropy pairs $\EP(Y)$, and show that $(y,z) \notin \Delta_Y$ is an entropy pair if and only if the set
\[ U = \{ u \in A^* \;|\; u \sqsubset y \vee u \sqsubset z \}, \]
containing all non-$\#$ subwords from $y$ and $z$, is a friendship.

Suppose $u', v' \in \hat A^{**}$ are any two positioned words that appear in some points of $Y$, with $\#$ as their first and last symbols (note that words beginning and ending with $\#$ are dense in $Y$ in the natural topology of $\hat A^{**} \cup \hat A^\Z$), and let $U$ be the set of words containing $u \in A^*$ iff $\#u\# \sqsubset u'$ or $\#u\# \sqsubset v'$. Suppose $U$ is a friendship. Extend $u'$ and $v'$ to full positioned words whose preskeletons are the same length, in an arbitrary way, though so that their finite subwords over $A^*$ remain a friendship (safe continuations for words can always be found in the bracelets proving the friendship). The left and right ends of the resulting words $u''$ and $v''$ have some left and right offsets $n_1, n_2$ respectively, in the sense that $u'' \in A^{[a,b]}$ and $v'' \in A^{[a+n_1,b+n_2]}$.

Now, continue extending the words by alternately turning them into full positioned words by adding a left and a right full extension, so that no new occurrence of $\#\#$ appears, and all finite subwords over $A^*$ remain a friendship. Note that the left and right ends continue to differ by exactly $n_1$ and $n_2$, since the length of the left or right continuation is always the same if no occurrence of $\#\#$ appears. Extend the words so many times that you have at least $\max(|n_1|, |n_2|)$ undoubled $\#$s on both sides. Now finally double some of the $\#$ to get two positioned full words $u, v$ which have $u'$ and $v'$ at the center, have the same preskeleton, and have the same left and right length, that is, $u, v \in A^{[c,d]}$ for some $c, d \in \Z$.

Now suppose $u$ and $v$ have preskeleton $w_i$. By the analysis of the density of $1$s in the ruler sequence, a word of length $k$ can contain at least $k/(6 \cdot 2^i) - 2$ occurrences of $u$ and $v$, which are \emph{interchangeable}, meaning we can change any subword $u$ into $v$ in any word $w \sqsubset Y$ containing it, to obtain another valid $w' \sqsubset Y$ (or vice versa). To see the density claim, recall there can be $k/3$ symbols $1$ in a word of length $k$ from the ruler sequence, and thus at least $k/(3 \cdot 2^i) - 2$ occurrences of $w_i$. We additionally precompose with $k \mapsto k/2$ since some of the $1$s are doubled in the skeleton, compared to the preskeleton.

The interchangeability, together with the lower bound on maximal density, shows that the cover $([u]^c, [v]^c)$ generates entropy in $Y$, and thus we have shown that any points whose non-$\#$ subwords form a friendship are entropy pairs. On the other hand if the union $U \subset A^*$ of the sets of subwords of two words $u, v$ is not a friendship, then these words do not even appear together in any point of $Y$, by definition, so neither can $u$ and $v$, and it follows that the cover $([u]^c, [v]^c)$ has zero entropy. This concludes the characterization of $\EP(Y)$.

An important corollary of this characterization is that $\EP(Y)$ is doubly shift-invariant, and therefore so is any $\EP(Y)^{(\lambda)}$ by an obvious transfinite induction. If $(x, y) \in F$ and $F$ is a doubly shift-invariant closed relation, then $(x', y) \in F$ for any $x'$ in $\OC{x}$, in particular this holds for the relations $F = \EP(Y)^{(\lambda)}$.

It is clear from the characterization that $\EP(Y) \cap X^2 = E$ as required. We need show $\EP(Y)^{(\lambda)} \cap X^2 = E^{(\lambda)}$ for all $\lambda$, and we do this by transfinite induction. On limit ordinal steps, this obviously continues to hold. On even successor steps, i.e.\ transitive closure steps, use the fact that points of $Y$ can be replaced by any point in their orbit closure, and the orbit-closure of every point contains a point of $X$. This allows turning any transitivity sequence $(x = y_0, y_1, ..., y_k = x')$ where $x, x' \in X$, $y_i \in Y$, $(y_i, y_{i+1}) \in E^{(\lambda-1)}$, into one where $y_i \in X$ for all $i$. This shows $(\EP(Y)^{(\lambda-1)})^* \cap X^2 = (\EP(Y)^{(\lambda-1)} \cap X^2)^*$ as required.

Now, note that if $\EP(Y)^{(\lambda)}$ is transitively closed, $y \in Y \setminus X$ and $b$ is any bracelet of $y$, or more generally any point in $X$ whose set of $A^*$-subwords together with $A^*$-subwords of $Y$ forms a friendship, then we have $(y, z) \in \EP(Y)^{(\lambda)}$ if and only if $(b, z) \in \EP(Y)^{(\lambda)}$. This is because $(b, y) \in \EP(Y)$ by the characterization of entropy pairs of $Y$. In other words, on every transitive step, every $y \in Y$ satisfies the same relations as any of its bracelets or any point that could be chosen as a bracelet for it.

On a topological closure step, i.e.\ for an odd ordinal $\lambda$, suppose $\lim_i (x_i, y_i) = (x, y) \in X^2$ where $(x_i, y_i) \in \EP(Y)^{(\lambda-1)}$ for all $i$. It is enough to show that the $x_i$ can be replaced by points from $X$, as then (also applying this idea to the right component) we have $\overline{\EP(Y)^{(\lambda-1)}} \cap X^2 = \overline{\EP(Y)^{\lambda-1} \cap X^2}$. If there is a subsequence of the $x_i$ in $X$, then we can directly restrict to that. Otherwise, restrict to a subsequence such that $x_i \notin X$ for all $i$.

In $x_i$, take a maximal central positioned word $u_i : A^{[-a, a]}$ not containing $\#$ (where necessarily $a \rightarrow \infty$ as $i \rightarrow \infty$) and replace $x_i$ by the bracelet $b_{u_i}$ with respect to the language of $x_i$, centered to have $u_i$ in the same place as in $x_i$. We have $(b_{u_i}, y_i) \in \EP(Y)^{(\lambda-1)}$ since $\EP(Y)^{(\lambda-1)}$ is transitively closed and $(b_{u_i}, x_i) \in \EP(Y)$, and clearly $(b_{u_i}, y_i) \rightarrow (x,y)$, as required. 

Now suppose $E^{(\lambda)} = X^2$ for some $\lambda$. If $\lambda$ is even, then $\EP(Y)^{(\lambda)}$ is transitively closed so since any $y \in Y \setminus X$ appears in the same relations as any bracelet of its, $\EP(Y)^{(\lambda)} = Y^2$. Suppose then that $\lambda$ is odd. Let $(y, z) \in Y^2$, and suppose $y, z \in Y \setminus X$, the other three cases being similar.\footnote{If $y \in X$, you can just pick all bracelets from the orbit of $y$, similarly for $z$, and you can approximate $y$ and $z$ in the exact same fashion. However, then $u_1\#u_2\#\cdots\#u_k$ does not look correct.} Let $(b_u)_u$ and $(c_v)_v$ be the respective bracelets of their subwords w.r.t.\ their languages. If any two of these bracelets, one from $y$ and one from $z$, are in relation on step $\lambda-1$, then already $(y, z) \in \EP(Y)^{(\lambda-1)}$ since $\EP(Y)^{(\lambda-1)}$ is transitive.

Otherwise, let $k$ be arbitrary and let (up to swapping $y$ and $z$) $m \geq k$ and $n \geq m$ be such that $y_{[-m,m]} = u_1\#u_2\#\cdots\#u_k$, $U = (u_1, ..., u_k) \in (A^*)^k$ and $z_{[-m,n]} = v_1\#v_2\#\cdots\#v_k$, $V = (v_1, ..., v_k) \in (A^*)^k$. Let $B = (b_u)_{u \in U}$ and $C = (c_v)_{v \in V}$ be bracelets for these words with respect to the subwords of $y$ and $z$, respectively.

Now, recall that the bracelets of a point are always (by definition) in the $E$-relation together, and apply the odd condition to $(B, C) \in (E^{(k)})^2$. Since $(E^{[k]})^2 \subset \overline{(E^{(\lambda-1)})^{[2k]} \cap (E^{[k]})^2}$, we can find in $(E^{(\lambda-1)})^{[2k]} \cap (E^{[k]})^2$ arbitrarily good approximations $B' = (b_u')_{u \in U}$ and $C' = (c_v')_{v \in V}$ to $B$ and $C$ respectively. In particular, we can find good enough approximations to ensure that $b_u'$ contains $u$ and $c'_v$ contains $v$, for each $u \in U, v \in V$. Then since $\EP(Y)^{(\lambda-1)}$ is transitive and $(b_u', c'_v) \in E^{(\lambda-1)}$, any pair of points $(y_k, z_k)$ in $Y$ for which we can take $b_u'$ and $c'_v$ respectively as one of their bracelets, satisfies $(y_k, z_k) \in \EP(Y)^{(\lambda-1)}$.

We construct the pair $y_k, z_k$ so that $(y_k, z_k)_{[-k,k]} = (y,z)_{[-k,k]}$. For this, directly let $(y_k)_{[-m,m]} = u_1\#u_2\#\cdots\#u_k$ and $(z_k)_{[-m,n]} = v_1\#v_2\#\cdots\#v_k$. Extend $y_k$ to a point of $Y$ by copying the skeleton from $y$ and filling all the remaining gaps arbitrarily so that all words between two $\#$s (and possible infinite tails without an occurrence of $\#$) are subwords of the words in $B'$. Extend $z_k$ similarly. These are indeed valid points: the skeletons are correct since we copied them from valid points, and the set of $A^*$-subwords of each point is a friendship, since $B', C' \in E^{[k]}$. Clearly we can take any $b_u'$ and $c'_v$ respectively as one of their bracelets, so $(y_k, z_k) \in \EP(Y)^{(\lambda-1)}$ gives an approximation to $(y,z)$ that is correct in the interval $[-k,k]$.
\end{proof}

\subsection{Implementation of arbitrary CPE ranks}

For the main proof, we make some simple topological observations.

\begin{lemma}
\label{lem:Open}
Suppose $X, Y$ are first-countable topological spaces, and $f : X \to Y$ is a quotient map. Then the following are equivalent:
\begin{itemize}
\item $f$ is open,
\item whenever $\lim_i y_i = y \in Y$ and $f(x) = y$, then there exist $x_i \in X$ such that $f(x_i) = y_i$ and $\lim_i x_i = x$.
\end{itemize}
\end{lemma}

\begin{proof}
Suppose $f$ is open and $f(x) = y$. Then every neighborhood $U \ni x$ maps to a neighborhood $f(U)$ of $y$, in particular if $\lim_i y_i = y$, then for all large enough $i$, $y_i$ has a preimage $x_i$ in $U$. Picking open sets $U$ from a decreasing countable neighborhood basis of $x$, and always picking preimages for $y_i$ from $U$ until they start having preimages in the next neighborhood of $x$, the claim is proved.

For the other direction, suppose $f$ is not open, and let $U$ be open such that for some $y \in f(U)$, every neighborhood of $y$ contains a point without a preimage in $U$. Pick $x \in U$ such that $f(x) = y$. By first-countability there is a sequence $y_i \rightarrow y$ such that the points $y_i$ have no preimage in $U$. The second condition fails, since we cannot pick $x_i \in U$.
\end{proof}

The following lemmas explain why we want $\phi$ to be open.

\begin{lemma}
\label{lem:OpenMapProcess}
Let $X,Y$ be first-countable topological spaces. Let $f : X \to Y$ be an open quotient map, $E \subset Y^2$ a symmetric reflexive closed relation, and $F = (f^{-1} \times f^{-1})(E)$. Then $F^{(\lambda)} = (f^{-1} \times f^{-1})(E^{(\lambda)})$ for all ordinals $\lambda$.
\end{lemma}

\begin{proof}
Since $X$ is first-countable, so is $X^2$, and thus the closure of a set $C \subset X^2$ is just the set of limits of sequences in $C$.\footnote{In other words $X^2$ has the \emph{Fr\'echet-Urysohn} property. For general topological spaces, even if $X$ is Frech\'et-Urysohn, $X^2$ need not be \cite{Gr06}.}

Write $g = f \times f : X^2 \to Y^2$, and observe this is also an open quotient map. We have $F^{(\zeta)} = g^{-1}(E^{(\zeta)})$ where $\zeta = 0$ if $E$ is not closed, and $\zeta = 1$ otherwise, since $E$ is closed if and only if $F$ is closed since $g$ is quotient map. We proceed by induction. Suppose $\lambda$ is an odd ordinal, and $g(x,x') = (y,y')$. If $(x,x') \in F^{(\lambda)}$ then there exist $(x_i,x_i') \in F^{(\lambda-1)}$ with $(x_i,x_i') \rightarrow (x,x')$ and then $g(x_i,x_i') \rightarrow (y,y')$ by continuity showing $E^{(\lambda)} \supset g(F^{(\lambda)})$, which implies $g^{-1}(E^{(\lambda)}) \supset F^{(\lambda)}$. 

For the other inclusion, suppose $(y, y') \in E^{(\lambda)}$, so $(y, y') = \lim_i (y_i, y_i')$ for some $(y_i, y_i') \in E^{(\lambda-1)}$. By the previous lemma, there exist $(x_i,x_i') \in X^2$ such that $g(x_i,x_i') = (y,y')$ and $(x_i, x_i') \rightarrow (x, x')$. Then by induction we have $(x_i,x_i') \in F^{(\lambda-1)}$ for all $i$, so $(x,x') \in F^{(\lambda)}$, showing $g^{-1}(E^{(\lambda)}) \subset F^{(\lambda)}$.

If $\lambda$ is an even successor ordinal, suppose $g(x,x') = (y,y')$. If $(x,x') \in F^{(\lambda)}$, there exist $x = x_0, x_1, ..., x_k = x'$ such that $(x_i, x_{i+1}) \in F^{(\lambda-1)}$ for all $i$, and then by induction $g(x_i, x_{i+1}) \in E^{(\lambda-1)}$ and since $y = f(x_0), y' = f(x_k)$ we have $(y,y') \in E^{(\lambda)}$.

If on the other hand $(y, y') \in E^{(\lambda)}$, then there exist $y = y_0, y_1,..., y_k = y'$ with $(y_i, y_{i+1}) \in E^{(\lambda-1)}$. Pick any preimages $f(x_i) = y_i$ with $x_0 = x, x_k = x'$. by induction $(x_i, x_{i+1}) \in F^{(\lambda-1)}$ for all $i$, and thus $(x,x') \in F^{(\lambda)}$.

For the limit ordinal case, observe that union commutes with preimage.
\end{proof}

\begin{lemma}
Let $X, Y$ be first-countable. Let $f : X \to Y$ be an open quotient map, $E \subset Y^2$ a symmetric reflexive relation, and $F = (f^{-1} \times f^{-1})(E)$. Then
\[ (E^{[k]})^2 \subset \overline{(E^{(\lambda-1)})^{[2k]} \cap (E^{[k]})^2} \]
if and only if
\[ (F^{[k]})^2 \subset \overline{(F^{(\lambda-1)})^{[2k]} \cap (F^{[k]})^2} \]
\end{lemma}

\begin{proof}
Suppose $(E^{[k]})^2 \subset \overline{(E^{(\lambda-1)})^{[2k]} \cap (E^{[k]})^2}$ and consider $(A', B') \in (F^{[k]})^2$. Consider $f(A')$ and $f(B')$ (where $f$ is applied diagonally to all elements of the tuple), which are by definition $k$-tuples $A, B \in E^{[k]}$. In $(E^{(\lambda-1)})^{[2k]} \cap (E^{[k]})^2$ we find arbitrarily good approximations $(A_i, B_i)$ to $(A,B)$. Take any $f$-preimages, $f(A_i') = A_i$, $f(B_i') = B_i$ which are close to $A', B'$ elementwise, using Lemma~\ref{lem:Open}, obtaining by definition $A'_i, B_i' \in F^{[k]}$. By the previous lemma, we have $(F^{(\lambda-1)})^{[2k]} = (f \times f)^{-1}(E^{(\lambda-1)})^{[2k]}$ so we have $(A_i', B_i') \in (F^{(\lambda-1)})^{[2k]}$. This implies $(F^{[k]})^2 \subset \overline{(F^{(\lambda-1)})^{[2k]} \cap (F^{[k]})^2}$.

The other direction is similar, but using continuity instead of openness.
\end{proof}

\main*

\begin{proof}
If $\alpha = 1$, $X = A^\Z$ is an example.

Let $\lambda \geq 2$ be a countable ordinal. By Lemma~\ref{lem:StrongTopVersion}, there exists a closed countable set $Z \subset [0,1]$ and a closed symmetric reflexive relation $E \subset Z^2$ such that $E$ is equalizing with close-up rank $\lambda$. Furthermore, if $\lambda$ is an odd ordinal, then $(E^{[k]})^2 \subset \overline{(E^{(\lambda-1)})^{[2k]} \cap (E^{[k]})^2}$.

By Lemma~\ref{lem:SpacesAsSubshifts}, since every countable compact subspace of $[0,1]$ is zero-dimensional, there exists a pointwise zero-entropy Toeplitz subshift $X$ and a doubly shift-invariant closed equivalence relation $K \subset X^2$ such that $Z$ is the open quotient of $X$, by $\phi : X \to Z$, and $\phi(x) = \phi(y) \iff (x,y) \in K$.

Since $\phi$ is an open quotient map, Lemma~\ref{lem:OpenMapProcess} shows that $E^{(\lambda)} = Z^2$ if and only if $(\phi^{-1} \times \phi^{-1})(E) = F$ satisfies $F^{(\lambda)} = X^2$. If $\lambda$ is an odd ordinal, then $(E^{[k]})^2 \subset \overline{(E^{(\lambda-1)})^{[2k]} \cap (E^{[k]})^2}$, and by the previous lemma we have $(F^{[k]})^2 \subset \overline{(F^{(\lambda-1)})^{[2k]} \cap (F^{[k]})^2}$.

Now, apply Lemma~\ref{lem:EqRelsAsEntPairs} to obtain $Y \supset X$ such that $\lambda$ is the minimal ordinal satisfying $\EP(Y)^{(\lambda)} = Y^2$. This $Y$ is a subshift in CPE class $\lambda$.
\end{proof}

\subsection{Not every relation satisfies odd condition}

We show that not every relation satisfies the odd condition, at least for one-step processes.

\begin{example}
\label{ex:NotEvery}
Let $X = [0,3]$. For $i,j \in \{0,1,2\}$, let $X_{i,j} \subset [j, j+1]$ be a dense set such that $X_{i,j} \cap X_{i',j'} = \emptyset$ for $(i,j) \neq (i',j')$.

Indexing modulo $3$, let $R_i = (X_{i,i} \cup X_{i,i+1})^2$. Now clearly $E = R_0 \cup R_1 \cup R_2$ is reflexive, symmetric and transitive but is not closed. The close-up rank is $1$: $E^{(0)} = E$ and $E^{(1)} = [0,3]^2$.

The odd condition $(E^{[2]})^2 \subset \overline{E^{[4]} \cap (E^{[2]})^2}$ fails because $(E^{[2]})^2$ contains quadruples $(x, y, z, w) \in [0,1] \times [1,2] \times [1,2] \times [2,3]$, while $E^{[4]}$ only contains quadruples contained in $[0,2]^4 \cup [1,3]^4 \cup ([0,1] \cup [2,3])^2$. \qee
\end{example}

\bibliographystyle{plain}
\bibliography{../../../bib/bib}{}

\end{document}